\documentclass[reqno, 12pt]{amsart}
\usepackage[utf8]{inputenc}

\usepackage{amsfonts,color,newclude}
\usepackage{amsthm}
\usepackage{amssymb,transparent}
\usepackage{amsmath}
\usepackage{amscd}
\usepackage{thmtools}
\usepackage{thm-restate}
\usepackage{setspace}
\usepackage{tikz-cd}

\usepackage{todonotes}

\usepackage{graphicx}
\usepackage{svg}
\usepackage{subfiles}
\usepackage{tikz}
\usepackage{pgffor} 
\usetikzlibrary{arrows, positioning, calc}
\usepackage{tikz-cd}
\usepackage{MnSymbol}

\usepackage[colorlinks=true, pdfstartview=FitV, linkcolor=blue, citecolor=blue, urlcolor=blue]{hyperref}
\usepackage{hyperref}
\usepackage[capitalize,nameinlink,noabbrev]{cleveref}

\renewcommand{\hat}{\widehat}

\newcommand{\Z}{\mathbb{Z}}

\declaretheorem[style = plain, name=Theorem, within=section]{theorem}
\declaretheorem[style = plain, name=Proposition,sibling=theorem]{prop}

\declaretheorem[style = plain, name=Lemma,sibling=theorem]{lemma}
\declaretheorem[style = plain, name=Corollary,sibling=theorem]{cor}

\theoremstyle{definition}
\newtheorem{defn}[theorem]{Definition}

\newtheorem*{centerconjecture}{The Center Conjecture}

\theoremstyle{remark}

\title{Centers of Artin Groups defined on cones}
\author{Kasia Jankiewicz}
\author{MurphyKate Montee}
\address{Department of Mathematics, University of California, Santa Cruz}
\email{kasia@ucsc.edu}
\address{Department of Mathematics and Statistics, Carleton College}
\email{mmontee@carleton.edu}

\begin{document}

\begin{abstract}
    We prove that the Center Conjecture passes to the Artin groups whose defining graphs are cones, if the conjecture holds for the Artin group defined on the set of the cone points. In particular, it holds for every Artin group whose defining graph has exactly one cone point.
\end{abstract}

\maketitle

\section{Introduction}
An Artin group $A$ is given by the presentation 
$$A = \langle s_1, \dots, s_n| \underbrace{s_is_js_i\cdots}_{m_{ij} \hspace{.5mm} terms} = \underbrace{s_js_is_j\cdots}_{m_{ij} \hspace{.5mm} terms} \rangle$$
where $m_{ij} = m_{ji} \geq 2$. The data of an Artin group can be encoded by its \emph{defining graph} $\Gamma$ whose vertex set is $V(\Gamma)=\{s_1, \dots, s_n\}$, and each relation involving $s_i, s_j$ with terms of length $m_{ij}$ corresponds to an edge $(s_i, s_j)$ with label $m_{ij}$. The Artin group with defining graph $\Gamma$ will be denoted by $A_\Gamma$.
Every Artin group has a naturally associated Coxeter group quotient $A_\Gamma \to W_\Gamma$ obtained by adding relations $s_i^2 = 1$ for all $i=1,\dots, n$.

An Artin group $A$ is \emph{spherical} if the corresponding Coxeter group is finite, and otherwise $A$ is \emph{infinite type}.
A \emph{special subgroup} of $A$ is a subgroup generated by some subset of $K\subseteq V(\Gamma)$, denoted $A_K$. Each special subgroup $A_K$ is itself isomorphic to the Artin group with defining graph $\Gamma_K$, where $\Gamma_K$ is the subgraph of $\Gamma$ induced by $K$ \cite{vanderLek83}.

A \emph{$2$-labelled join} of labelled graphs $\Gamma_1, \Gamma_2$ is a graph join $\Gamma_1*\Gamma_2$ where every edge $(s_1,s_2)$, with $s_1\in V(\Gamma_1)$ and $s_2\in V(\Gamma_2)$, has label $2$, and the labels of edges contained in factors $\Gamma_1, \Gamma_2$ remain the same. We denote a 2-labelled join by $\Gamma_1 *_2 \Gamma_2$. 
If $\Gamma = \Gamma_1 *_2 \Gamma_2$, then $A_\Gamma = A_{\Gamma_1}\times A_{\Gamma_2}$.
An Artin group $A$ is \emph{irreducible} if its defining graph $\Gamma$ does not split as a $2$-labelled join, in which case we also call $\Gamma$ irreducible.

Each Artin group $A_\Gamma$ admits (a possibly trivial) \emph{decomposition into irreducible factors} $A_{\Gamma}= A_{\Gamma_1}\times \dots \times A_{\Gamma_n}$, where each $\Gamma_i\subseteq \Gamma$ is irreducible. Equivalently, this corresponds the maximal decomposition of $\Gamma$ as an $n$-fold $2$-labelled join of subgraphs $\Gamma_1, \dots, \Gamma_n$. 

Every irreducible spherical Artin group $A$ has an infinite cyclic center, and furthermore if $\{s_1, \dots, s_n\}$ is the standard generating set of $A$ then $Z(A)$ is generated by a power of $s_1s_2\cdots s_n$ \cite{Deligne72, BrieskornSaito72}.
If $z$ is a central element in a spherical factor $A_{\Gamma_i}$ of $A_\Gamma$, then $z\in Z(A_\Gamma)$.
Conjecturally, all the central elements of $A_\Gamma$ arise from irreducible factors that are spherical.

\begin{centerconjecture}
    Let $A_\Gamma$ be an Artin group with irreducible factor decomposition
    \[A_\Gamma = A_{\Gamma_1}\times \dots \times A_{\Gamma_n}.\]
    Then $Z(A_\Gamma) = \mathbb Z^k$ where $k$ is the number of spherical factors $A_{\Gamma_i}$.
\end{centerconjecture}

When $A = B\times C$ then $Z(A) = Z(B) \times Z(C)$ so, equivalently, the Center Conjecture states that every irreducible Artin group of infinite type has trivial center.

Apart from the spherical Artin groups, the Center Conjecture also holds for FC-type Artin groups, $2$-dimensional Artin groups \cite{GodelleParis12a}, and Euclidean Artin groups \cite{McCammondSulway}. More generally, every Artin group that satisfies the $K(\pi,1)$-conjecture also satisfies the Center Conjecture \cite{JankiewiczSchreve22}. 

A \emph{cone point} in a graph $\Gamma$ is a vertex that is adjacent to every other vertex of $\Gamma$. Charney and Morris-Wright have shown the Center Conjecture holds for Artin groups whose defining graphs are not cones \cite{CharneyMorrisWright19}. In other words, this is the case where the set of cone points of $\Gamma$ is empty. The proof of Charney and Morris-Wright relies on the geometry of the clique-cube complex, which is a CAT(0) cube complex associated with $\Gamma$. Their result can also be deduced from the following proposition, whose proof is brief and follows directly from the presentation of $A_\Gamma$.

\begin{restatable}{proposition}{conepointstheorem}\label{thm intro: center contained in cone point spherical}
    Let $A_\Gamma$ be an Artin group, where $T\subseteq V(\Gamma)$ is the set of cone points of $\Gamma$. Then $Z(A_{\Gamma})\subseteq Z(A_T)$.
\end{restatable}

In particular, if in the above theorem $Z(A_{T}) = \{1\}$, then $Z(A_\Gamma) = \{1\}$, i.e.\ $A_{\Gamma}$ satisfies the Center Conjecture. 

Godelle and Paris showed that if all Artin groups whose defining graph is a clique satisfy the Center Conjecture, then all Artin groups satisfy the Center Conjecture \cite{GodelleParis12a}. Our second result gives alternative proofs of this fact and the Center Conjecture for FC-type Artin groups, as well as the Center Conjecture for many new Artin groups (see \Cref{fig:examples} for some examples).

\begin{restatable}{theorem}{maintheorem}\label{thm intro: center conjectures passes to cones}
    Let $A_\Gamma$ be an Artin group, where $T\subseteq V(\Gamma)$ is the set of cone points of $\Gamma$. If $A_{T}$ satisfies the Center Conjecture, then $A_\Gamma$ satisfies the Center Conjecture.
\end{restatable}

\begin{figure}
    \centering
    \includegraphics[width = \textwidth]{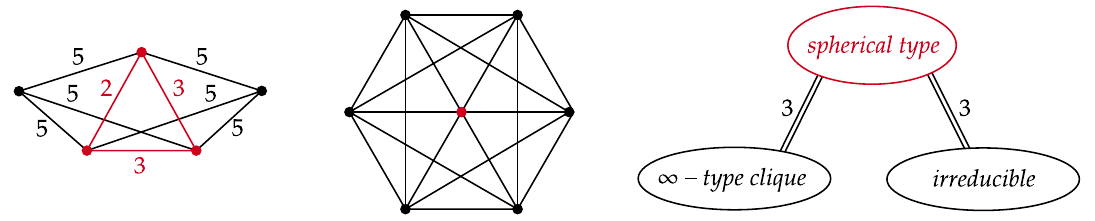}
    \caption{New examples of graphs whose associated Artin groups satisfy the Center Conjecture. The red subgraphs are spanned by the set of cone points. The middle graph can have any labels. The diagram on the right represents an infinite class of such graphs; the double edges represent joins, with all labels equal to 3.}
    \label{fig:examples}
\end{figure}

To prove \Cref{thm intro: center contained in cone point spherical} we consider splittings of non-clique Artin groups. This generalizes the result of \cite{CharneyMorrisWright19} using a different method.
Then to deduce \Cref{thm intro: center conjectures passes to cones} from \Cref{thm intro: center contained in cone point spherical} we use a retraction map (first used in \cite{CharneyParis14} and recently discussed more explicitly in \cite{BlufsteinParis23}) to prove some combinatorial restrictions on central elements of irreducible Artin groups; namely that any spelling of a central element which does not use every generator cannot be strictly positive or strictly negative.

\subsection*{Acknowledgements} The first author was partially supported by NSF grants DMS-2203307 and DMS-2238198.  The second author was partially supported by NSF grant number DMS-2317001.

\section{Central Elements are Generated by Cone Points}\label{sec:cone points}

Let $\Gamma$ a graph and let $x, y \in V(\Gamma)$ so that $m_{x,y} = \infty$.  A direct consequence of the presentation of $A_\Gamma$ is that it splits as the following amalgamated product:
    \[
        A_\Gamma = A_{\Gamma - \{x\}} *_{A_{\Gamma - \{x, y\}}} A_{\Gamma - \{y\}}.
    \]

The following lemma can be proven by considering normal forms in amalgamated free products, and was utilized in \cite{GodelleParis12a}.

\begin{lemma}\label{lem: center of free product}
    Let $G = H_1 *_K H_2$ be the free product of $H_1, H_2$ amalgamated over $K$. Then $Z(G) = Z(H_1) \cap Z(H_2) \subseteq K$. 
\end{lemma}

Using these two results we can prove our first theorem.

\conepointstheorem*

\begin{proof}
    Let $R = V(\Gamma) - T$ and $r \in R$. There exists some $q_r \in V(\Gamma) - T$ so that $m_{r, q_r} = \infty$. Thus  
        \[
            A_\Gamma = A_{\Gamma-\{r\}} *_{A_{\Gamma - \{r, q_r\}}} A_{\Gamma - \{q_r\}}.
        \]
    By \Cref{lem: center of free product} this implies that for every $r \in R$ we have $Z(A_\Gamma) \subseteq {A_{\Gamma - \{r, q_r\}}} \subseteq A_{\Gamma - \{r\}}.$ In particular we have 
        \[
            Z(A_\Gamma) \subseteq \bigcap_{r \in R} A_{\Gamma - \{r\}}.
        \]
    
    In \cite{vanderLek83} van der Lek showed that for any $S, Q \subseteq V(\Gamma)$ we have $A_S \cap A_Q = A_{S\cap Q}$. In particular, $\bigcap_{r \in R} A_{\Gamma - \{r\}} = A_{\Gamma - R} = A_T$. Since $z \in Z(A_\Gamma)$, $z$ is also central in $A_T$.
\end{proof}

We obtain the following corollary immediately. 

\begin{cor}
    Suppose that $\Gamma$ is a graph with cones points $T$, and suppose $A_T$ has trivial center. Then $A_\Gamma$ has trivial center.
\end{cor}

In particular, if $T = \emptyset$ then $Z(A_\Gamma) = \{1\}$ and we recover \cite[Thm 3.3]{CharneyMorrisWright19}.  The condition that $Z(A_T) = \{1\}$ is not a necessary condition: if $T = \{t\}$ then $A_T \cong \Z$, so $A_T$ does not have trivial center. On the other hand, if there is any vertex of $\Gamma$ which does not commute with $t$ then $Z(A_\Gamma)$ is trivial.

\begin{cor}
    Suppose $\Gamma$ is a graph with a single cone point $t$. Then $A(\Gamma)$ satisfies the Center Conjecture, i.e.\ $Z(A_\Gamma) = \{1\}$ if and only if there exists $s \in S$ so that $m_{st} \neq 2$. . 
\end{cor}
    \begin{proof}
        By \Cref{thm intro: center contained in cone point spherical} $z \in A_{\{t\}} = \langle t \rangle$. Let $z = t^i$. Since $A_{\{t, s\}}$ is spherical, the center of $A_{\{t, s\}}$ is generated by a power of $ts$. Hence $z$ is central in $A_{\{t, s\}}$ only if $i = 0$.
        
        If $m_{st} = 2$ for all $s \in S - \{t\}$, then $A_\Gamma \cong \langle t \rangle \times A_S$, so $\langle t \rangle \leq Z(A_\Gamma)$. 
    \end{proof}

In the following section we generalize this argument to larger sets of cone points.

\section{Cone Sets Satisfying the Center Conjecture}\label{sec: retraction}

\subsection{Retraction map}
The main goal of this subsection is \Cref{thm:central elements in cones}. Its proof relies on the retraction map $\pi_X :A \to A_X$ described in \cite{CharneyParis14, BlufsteinParis23}. The definition of the map $\pi_X$ relies on passing to the Coxeter group $W_\Gamma$. We re-establish notation here so that we can differentiate easily between elements in the Coxeter group and elements in the Artin group. 

Fix a graph $\Gamma$. Let $\Sigma = \{\sigma_v | v \in V(\Gamma)\}$ and $S = \{s_v | v \in V(\Gamma)\}$ be generating sets for the Artin and Coxeter groups $A_\Gamma$ and $ W_\Gamma$, respectively. Let $\theta: A_\Gamma \to W_\Gamma$ the natural surjection sending $\sigma_v \mapsto s_v$. The kernel of $\theta$ is the \emph{pure Artin group} $P\!A_\Gamma$. Let $X \subseteq V(\Gamma)$. We denote the subgroup of $W_\Gamma$ generated by $X$ by $W_X$, and similarly for the corresponding subgroup of $P\!A_\Gamma$ and subsets of $\Sigma, S$.

Let $X, Y$ be two subsets of $S$. An element $w \in W_\Gamma$ is \emph{$(X, Y)$-reduced} if it is of minimal length amongst the elements of the double coset $W_X w W_Y$. For any element $w \in W_\Gamma$ the \emph{length} of $w$, denoted $|w|$, is the shortest length of a word on $V(\Gamma)$ needed to express $w$. 

\begin{lemma}[Lemma 2.3 \cite{GodelleParis12}]\label{lemma: X-empty reduced}
    Let $X, Y \subseteq V(\Gamma)$ and let $w \in W_\Gamma$. There exists a unique $(X, Y)$-reduced element in $W_X w W_Y.$ Furthermore, the following are equivalent:
    \begin{itemize}
        \item an element $w \in W$ is $(X, \emptyset)$-reduced,
        \item $|sw|>|w|$ for all $s \in X$, and
        \item $|v\cdot w| = |v|+|w|$ for all $v\in W_X$.
    \end{itemize}
\end{lemma}

The map is defined as follows. 

\begin{defn}[\cite{BlufsteinParis23}]\label{defn:retraction}
Let $(\Sigma\cup \Sigma^{-1})^*$ denote the free monoid on $\Sigma\cup \Sigma^{-1}$. Let $X \subseteq V(\Gamma)$, and let $\hat{\alpha} = \sigma_{v_1}^{\varepsilon_1}\sigma_{v_2}^{\varepsilon_2} \cdots \sigma_{v_p}^{\varepsilon_p} \in (\Sigma\cup \Sigma^{-1})^*.$ 

Set $u_0 = 1 \in W_\Gamma$, and for $i \in \{1, \dots, p\}$ set $u_i = s_{v_1}s_{v_2}\cdots s_{v_i} \in W_\Gamma$. We can write each $u_i$ as $u_i = v_iw_i$ where $v_i \in W_X$ and $w_i$ is $(X, \emptyset)$-reduced. 

Now define    
    \[
    t_i = \begin{cases}
        w_{i-1}s_{v_i}w_{i-1}^{-1} &\mbox{ if } \varepsilon = 1, \\
        w_i s_{v_i}w_i^{-1} &\mbox{ if } \varepsilon = -1.
    \end{cases}
    \]
If $t_i \notin S_X$ we set $\tau_i = 1$. Otherwise, we have $t_i \in S_X$, and we set $\tau_i = \sigma_{x_i}^{\varepsilon_i}\in \Sigma\cup \Sigma^{-1}$.

Finally, we define
    \[
    \hat\pi_X(\hat{\alpha}) = \tau_1 \dots \tau_p \in (\Sigma\cup \Sigma^{-1})^*.
    \]
\end{defn}

We collect in the following proposition several properties of the map $\hat{\pi}_X$ and the induced map $\pi_X$ (see \Cref{prop: retraction properties}(\ref{subprop: map pi})) that will be of use in later results.

\begin{prop}\label{prop: retraction properties}
Let $X, Y \subseteq V(\Gamma)$, and $\alpha \in (\Sigma\cup\Sigma^{-1})^*$.
\begin{enumerate}
    \item \cite[Prop 2.3(1)]{BlufsteinParis23}\label{subprop: map pi} The map $\hat\pi_X: (\Sigma\cup \Sigma^{-1})^* \to (\Sigma_X\cup \Sigma^{-1}_X)^*$ induces a set-map $\pi_X: A_\Gamma \to A_X$.
    \item \cite[Prop 2.3(2)]{BlufsteinParis23}\label{subprop: id on X} For all $\alpha \in A_X$ we have $\pi_X(\alpha) = \alpha$.
    \item \cite[Prop 2.3(3)]{BlufsteinParis23}\label{subprop: hom on pure artin} The restriction of $\pi_X$ to $P\!A_\Gamma $ is a homomorphism $\pi_X:P\!A_\Gamma \to P\!A_X$.
    \item \cite[Prop 0.2(vi)]{Godelle23} \label{subprop: pi preserves positive}  Let $A_\Gamma ^+, A_X^+$ denote the Artin monoids on $V(\Gamma), X$, respectively. Then $\pi_X: A_\Gamma ^+ \to A_X^+$.
    \item \cite[Prop 0.2(ix)]{Godelle23} \label{subprop: restriction of retraction} If $\omega \in A_Y$ then $\pi_X(\omega) \in A_{Y\cap X}$.
    \item \cite[Prop 0.3(ii)]{Godelle23} \label{subprop: pi sends intial X to X} For any $\omega \in A_X$ we have $\pi_X(\omega\alpha) = \omega \pi_X(\alpha).$
\end{enumerate}
\end{prop}

\begin{theorem}\label{thm:central elements in cones}
    Let $A_\Gamma$ be an irreducible Artin group and let $z$ be either a central element of $A_\Gamma$ of infinite order, or a non-trivial central element of $P\!A_\Gamma$, such that $z$ admits a positive spelling in $V(\Gamma)$. There does not exist any $K \subsetneq V(\Gamma)$ so that $z \in A_K$.
\end{theorem}

\begin{proof}
    Suppose $z$ is central in $A_\Gamma$ and has infinite order, and suppose that $z \in A_K$ for some $K \subsetneq V(\Gamma)$.  Suppose further that $z$ is a positive word on $K$; that is, there exists an $\hat\alpha \in K^*$ so that $z = \hat{\alpha}$. Since $z$ is central, $\theta(z)$ lies in the center of $W_\Gamma$. In particular, it has finite order \cite{Bourbaki}. By possibly passing to a finite positive power of $z$, we can assume that $z \in Z(P\!A_\Gamma)$.

    Without loss of generality, by possibly passing to a smaller subset of $K$ we can assume that all the letters of $K$ appear in $\hat\alpha$. We claim that there exists $t \in K$ and $s \in V(\Gamma) - K$ so that $m_{st} \neq 2$. Indeed, if there is no such $t \in K$, then $\Gamma = K *_2 (\Gamma - K)$, 
    which contradicts the assumption that $A_\Gamma$ is irreducible. Let $\hat\alpha'$ denote a cyclic permutation of $\hat\alpha$ beginning with $t$. Note that every cyclic permutation of $\hat\alpha$ is a conjugate of $\hat\alpha$, and therefore $\hat\alpha$ represents $z$, since $z$ is central.  Consider the map $\pi_{\{s, t\}}: A_\Gamma \to A_{\{s, t\}}$. By \Cref{prop: retraction properties}(\ref{subprop: pi sends intial X to X}) we know that $\hat\pi_{\{s, t\}}(\hat\alpha')$ starts with the letter $t$, and by \Cref{prop: retraction properties}(\ref{subprop: pi preserves positive}) we know that $\hat\pi_{\{s, t\}}(\hat\alpha')$ is a positive word. In particular, $\pi_{\{s, t\}}(z)$ is non-trivial. 
    By \Cref{prop: retraction properties}(\ref{subprop: restriction of retraction}) $\pi(z)\in A_{K\cap \{s, t\}} = A_{\{t\}}$, and therefore $\pi(z)$ is a positive power of $t$. 

    By \cref{prop: retraction properties}(\ref{subprop: hom on pure artin}) the map $\pi_{\{s, t\}}:P\!A_\Gamma \to P\!A_{\{s, t\}}$ is a surjective homomorphism, so $\pi_{\{s, t\}}(z)$ is central in $P\!A_{\{s, t\}}$. But since $A_{\{s, t\}}$ is spherical, its center is generated by a power of $st$. In particular, no non-trivial power of $t$ commutes with any non-trivial power of $s$ in $A_{\{s,t\}}$. Therefore $z = 1.$
\end{proof}

In the case that $\Gamma$ is not a clique, together with \Cref{thm intro: center contained in cone point spherical} this implies that no central element of $A_\Gamma$ admits a positive spelling.

\subsection{Proof of \Cref{thm intro: center conjectures passes to cones}}
We are now ready to prove the main theorem of this note.

\maintheorem*

    \begin{proof}
        If $\Gamma$ is a clique, then $V(\Gamma) = T$, so by assumption $A_\Gamma$ satisfies the Center Conjecture. Suppose instead that $\Gamma$ is not a clique. 
        
        It suffices to prove the theorem for irreducible $\Gamma$. Indeed, if $\Gamma = \Gamma_1*_2 \dots *_2\Gamma_n$ then $Z(A_\Gamma) = Z(A_{\Gamma_1}) \times \dots \times Z(A_{\Gamma_n})$ so the Center Conjecture holds for $\Gamma$ if and only if it holds for each irreducible factor $\Gamma_i$. The set of cone-points $T$ of $\Gamma$ is the union of the sets $T_i = T\cap V(\Gamma_i)$ of cone points in each $\Gamma_i$.  
        Furthermore, if $T$ splits as $T = {T'_1}*_2 \cdots *_2 {T'_n}$ then each $T_i$ is a union of sets $T'_{i_1}, \dots, T'_{i_j}.$
        Hence $A_{T_i}$ satisfies the Center Conjecture for each $i$ because $A_T$ satisfies the Center Conjecture. 

        Suppose that $\Gamma$ is irreducible. By \cref{thm intro: center contained in cone point spherical} $Z(A_\Gamma) \subseteq Z(A_T)$. 
        Let $A_T = A_{T_1}\times \cdots \times A_{T_m}\times \cdots \times A_{T_n}$ be the irreducible factor decomposition of $A_T$, where the factor $A_{T_i}$ is spherical if and only if $i \leq m$. 
        Since $A_T$ satisfies the Center Conjecture, $Z(A_T) = \langle z_1, \dots, z_m\rangle\simeq \mathbb Z^m$ where $z_i$ is a positive element of $A(T_i)$ generating $Z(A_{T_i})$ such that any spelling of $z_i$ uses every letter of $T_i$ at least once.   
        Thus $z = z_1^{k_1} \cdots z_m^{k_m}$ for some $k_i \in \mathbb{Z}$ for $i=1,\dots, m$. 
        Since $z$ is non-trivial, we know that at least one of the $k_i$ is non-zero. 
        Up to renumbering the factors and replacing $z$ with $z^{-1}$, we may assume that $k_1>0$. 
        There is some $k>0$ so that $z^k \in P\!A_\Gamma.$
        Let $x = z_1^{kk_1}$ and $y = z_{2}^{kk_{j+1}}\cdots z_m^{kk_m}$ so that $z^k = xy$. 
        Let $\hat\alpha_x, \hat\alpha_y$ denote minimal spellings of $x, y$, respectively. 
        Note that $\hat\alpha_x$ is positive. 

        There must exist some $r \in V(\Gamma)-T, t \in T_1$ so that $m_{rt}\neq 2$. Indeed, if this is not the case then $\Gamma = (\Gamma-T_1) *_2 {T_1}$, contradicting the irreducibility of $\Gamma$. 
        Since all cyclic permutations of $\hat\alpha_x$ represent $x$, we may take $t$ to be the initial letter of $\hat\alpha_x\hat\alpha_y$.
        Now consider the map $\pi_{T_1 \cup \{r\}}: A_\Gamma \to A_{T_1 \cup \{r\}}$. We have $\pi_{T_1 \cup \{r\}}(z^k) \in Z(P\!A_{T_1 \cup \{r\}}).$

        By \Cref{prop: retraction properties}(\ref{subprop: pi sends intial X to X}) $\pi_{T_1 \cup \{r\}}(xy) = x\pi_{T_1 \cup \{r\}}(y)$, since $x\in A_{T_1\cup \{r\}}$. Since $y \in A_{T - \{T_1\cup\{r\}\}}$, by \Cref{prop: retraction properties}(\ref{subprop: restriction of retraction}) we have $\pi_{T_1 \cup \{r\}}(y) = 1$.        
        So $\pi_{T_1 \cup \{r\}}(z^k) = x$.
        Thus $x$ is a positive central element of $P\!A_{T_1 \cup \{r\}}$ 
        which does not use every letter of $T_1 \cup \{r\}$. By \Cref{thm:central elements in cones} $x=1$, contradicting the choice of $k_1$.
        
    \end{proof}

\bibliographystyle{alpha}
\bibliography{artin_centers.bib}

\end{document}